%% file: sn-article.tex
\documentclass[sn-mathphys]{sn-jnl}% Math and Physical Sciences Reference Style
\jyear{2022}%

\theoremstyle{thmstyleone}%
\newtheorem{theorem}{Theorem}%  meant for continuous numbers
\newtheorem{lemma}[theorem]{Lemma}% 

\theoremstyle{thmstyletwo}%

\theoremstyle{thmstylethree}%

\usepackage{hyperref}
\usepackage{academicons}
\usepackage{xcolor}

\newcommand{\orcid}[1]{\href{https://orcid.org/#1}{\textcolor[HTML]{A6CE39}{\aiOrcid}}}

\begin{document}

\title[Orbit growth of sofic shifts]{A short note on the orbit growth of sofic shifts}


\author*[1]{\fnm{Azmeer} \sur{Nordin}}\email{nfazmeer.nordin@siswa.ukm.edu.my}

\author[1]{\fnm{Mohd Salmi Md} \sur{Noorani}}\email{msn@ukm.edu.my}
\equalcont{This author contributed equally to this work.}

\affil[1]{\orgdiv{Department of Mathematical Sciences}, \orgname{Universiti Kebangsaan Malaysia}, \orgaddress{\city{Bangi}, \postcode{43600}, \state{Selangor}, \country{Malaysia}}}


\abstract{A sofic shift is a shift space consisting of bi-infinite labels of paths from a labelled graph. Being a dynamical system, the distribution of its closed orbits may indicate the complexity of the space. For this purpose, prime orbit and Mertens' orbit counting functions are introduced as a way to describe the growth of the closed orbits. The asymptotic behaviours of these counting functions can be implied from the analiticity of the Artin-Mazur zeta function of the space. Despite having a closed-form expression, the zeta function is expressed implicitly in terms of several signed subset matrices, and this makes the study on its analyticity to be seemingly difficult. In this paper, we will prove the asymptotic behaviours of the counting functions for a sofic shift via its zeta function. This involves investigating the properties of the said matrices. Suprisingly, the proof is rather short and only uses well-known facts about a sofic shift, especially on its minimal right-resolving presentation.}

\keywords{Sofic shift, prime orbit counting function, Mertens' orbit counting functions, Artin-Mazur zeta function, minimal right-resolving presentation}


\pacs[MSC Classification]{37C35, 37C30, 37B10}

\maketitle

\section{Introduction}

Let $X$ be a compact metric space equipped with a continuous map $T:X \rightarrow X$. For the discrete dynamical system $(X,T)$, recall that a point $x \in X$ is periodic with period $n \in \mathbb{N}$ if $T^n(x)=x$. Moreover, if $T^k(x) \neq x$ for $k = 1,2,\hdots,n-1$, then it has the prime period $n$. Its (prime) closed orbit is then defined as
$$\tau(x) = \left\{x, T(x), \hdots, T^{n-1}(x)\right\}.$$

In many instances, the distribution of closed orbits are closely related to the complexity of a system. For this reason, some counting functions are introduced as a way to describe the growth of the closed orbits. These are called the \emph{prime orbit counting function}
$$\pi_T(N)=\sum_{\substack{\tau \\ \lvert \tau \rvert \leq N}} 1$$
and the pair of \emph{Mertens' orbit counting functions}
$$\mathcal{M}_T(N)=\prod_{\substack{\tau \\ \lvert\tau\rvert\leq N}} \left(1-\frac{1}{e^{h\lvert\tau\rvert}}\right) \quad\textrm{and}\quad \mathscr{M}_T(N)=\sum_{\substack{\tau \\ \lvert\tau\rvert\leq N}}\frac{1}{e^{h\lvert\tau\rvert}}$$
where $N \in \mathbb{N}$, $h$ is the topological entropy of the system (which is assumed to be positive) and $\tau$ runs through the closed orbits of size $\lvert\tau\rvert$.

These functions arise as the analogues for the counting functions for primes in number theory. In particular, the prime number theorem and Mertens' theorem state that
$$\sum_{\substack{p \\ p\leq N}}1 \sim \frac{N}{\ln N}, \quad \prod_{\substack{p \\ p\leq N}}\left(1-\frac{1}{p}\right) \sim \frac{e^{-\gamma}}{\ln N} \quad\textrm{and}\quad \sum_{\substack{p \\ p\leq N}} \frac{1}{p} = \ln \ln N + M +o(1)$$
where $\gamma$ and $M$ are Euler-Mascheroni constant and Meissel-Mertens constant, respectively, and $p$ runs through primes (see \cite{Hardy-Wright}). These theorems motivate an analogous problem in the theory of dynamical systems, which is to obtain the asymptotic behaviours of the orbit counting functions for a system (see \cite{Parry-Pollicott}).

As an approach, the analysis on the Artin-Mazur zeta function \cite{Artin-Mazur} of a system may lead to the desired asymptotic results. For a system $(X,T)$, its Artin-Mazur zeta function is defined as
$$\zeta_T(z) = \exp \left(\sum_{n=1}^{\infty}\frac{F_T(n)}{n}z^n\right)$$
where $z \in \mathbb{C}$ (whenever it converges) and $F_T(n)$ is the number of periodic points of period $n$. The next theorem summarizes on how the analiticity of the zeta function implies the asymptotic behaviours of the orbit counting functions.
\begin{theorem}[\cite{Nordin-Noorani_PFT}]\label{theorem: zeta function approach}
	Let $(X,T)$ be a discrete dynamical system with topological entropy $h>0$ and Artin-Mazur zeta function $\zeta_T(z)$. Suppose that there exists a function $\alpha(z)$ such that it is analytic and non-zero for $\lvert z \rvert <Re^{-h}$ for some $R>1$, and
	\begin{equation*}
		\zeta_T(z)= \frac{\alpha(z)}{(1-e^{hp}z^p)^m}
	\end{equation*}
	for $\lvert z \rvert <e^{-h}$ for some $m, p \in \mathbb{N}$. Then,
	\begin{equation*}
		\pi_T(N) \sim  mp \cdot \frac{e^{hp\left(\left \lfloor \frac{N}{p} \right \rfloor +1\right)}}{(e^{hp}-1)N},
		\quad
		\mathcal{M}_T(N) \sim \frac{p^m e^{-m\gamma}}{\alpha(e^{-h}) \cdot N^m} \quad \textrm{and}
	\end{equation*}
	\begin{equation*}
		\mathscr{M}_T(N) = m \ln \left \lfloor \frac{N}{p} \right \rfloor + m \gamma + \ln \alpha(e^{-h}) - C +o(1)
	\end{equation*}
	where $\gamma$ is Euler-Mascheroni constant and $C$ is a positive constant specified as
	\begin{equation}\label{equation: C}
		C = \sum_{\tau} \left(\ln \left(\frac{1}{1-e^{-h\lvert\tau\rvert}}\right) - \frac{1}{e^{h\lvert \tau \rvert}}\right).
	\end{equation}
\end{theorem}
To sum up, we require $\zeta_T(z)$ to satisfy the following: (i) it extends to a meromorphic function for $\lvert z \rvert <Re^{-h}$, (ii) it has no zero in the said region, and (iii) there are exactly $p$ poles in this region, in which they have the same order $m$ and are of the form $\omega e^{-h}$, where $\omega$ runs through $p$th roots of unity.

In the literature, this approach was used to determine the orbit growths of ergodic toral automorphisms \cite{Noorani,Waddington} and several types of shift spaces. These include shifts of finite type \cite{Parry,Parry-Pollicott}, periodic-finite-type shifts \cite{Nordin-Noorani_PFT}, Dyck and Motzkin shifts \cite{Nordin-Noorani-Dzulkifli_DM}, and bouquet-Dyck shifts \cite{Nordin-Noorani_BD}. Similar results can be deduced for beta shifts \cite{Flatto-Lagarias-Poonen}, negative beta shifts \cite{NguemaNdong} and shifts of quasi-finite type \cite{Buzzi}, albeit these findings are not stated in their respective papers.

While we focus on the approach via zeta function, there are other methods to obtain the orbit growth of a system, such as using orbit Dirichlet series \cite{Everest-Miles-Stevens-Ward}, orbit monoids \cite{Pakapongpun-Ward} and estimates on the number of periodic points \cite{Akhatkulov-Noorani-Akhadkulov,Alsharari-Noorani-Akhadkulov_Dyck,Alsharari-Noorani-Akhadkulov_Motzkin}. Furthermore, similar research problem has been studied for group actions on dynamical systems, and some recent results include the orbit growths of nilpotent group shifts \cite{Miles-Ward}, algebraic flip systems \cite{Miles} and flip systems for shifts of finite type \cite{Nordin-Noorani_flip}. Since our introduction on this subject is rather short, we encourage readers to explore those papers above, and additionally the expository chapters by us \cite{Nordin-Noorani-Dzulkifli} and Ward \cite{Ward}.

Our attention now is on sofic shifts (see \cite{Lind-Marcus}). Sofic shifts are a class of shift spaces constructed from labelled graphs. Examples are shifts of finite type and periodic-finite-type shifts. The zeta function for these systems has long been known, and in fact, it is a rational function. However, it is implicitly expressed in terms of several signed subset matrices, making it rather sophisticated.

The orbit growth of sofic shifts is yet to be determined completely. A sofic shift with specification property (which is equivalent to topological mixing in this case) is a shift of quasi-finite type, so its orbit growth is deduced from the result in \cite{Buzzi} and Theorem \ref{theorem: zeta function approach} above. However, this finding does not account for irreducible sofic shifts.

Hence, our paper here aims to obtain the orbit growth, i.e. the asymptotic behaviours of the orbit counting functions, for more general sofic shifts. Suprisingly, the work by Lind \& Marcus \cite{Lind-Marcus}, especially on minimal right-resolving presentations of sofic shifts, is sufficient to lead to our finding. At the end, we provide a few remarks on the consequences of our finding.

\section{Sofic Shifts}

In this section, we provide some background on sofic shifts and their important properties. All these facts can found in \cite{Lind-Marcus}. 

Let $G=(V,E)$ be a finite directed graph with vertex set $V$ and edge set $E$. For an edge $e$, denote $i(e)$ and $t(e)$ as its initial and terminal vertices respectively. We assume that each vertex is neither a source or a sink, i.e. it must have incoming and outgoing edges. For a finite set $\mathcal{A}$, let $\mathcal{L}:E \rightarrow \mathcal{A}$ be the label of each edge with an element of $\mathcal{A}$. The pair $\mathcal{G}=(G, \mathcal{L})$ is called a \emph{labelled graph}.

A sequence $\{e_k\}_{k=1}^n$ of edges is called a path of length $n$ if $t(e_{k})=i(e_{k+1})$ for $k=1,2,\hdots,n-1$. The label of the path is the corresponding sequence $\left\{\mathcal{L}(e_k)\right\}_{k=1}^n$. We define a bi-infinite path $\{e_k\}_{k\in \mathbb{Z}}$ and its label $\left\{\mathcal{L}(e_k)\right\}_{k\in \mathbb{Z}}$ in similar way. Denote $\mathcal{P}_\infty(\mathcal{G})$ as the set of bi-infinite paths in $\mathcal{G}$.

Let $\mathcal{A}$ be equipped with the discrete topology. So, the set $\mathcal{A}^{\mathbb{Z}}$ of bi-infinite sequences from $\mathcal{A}$ is equipped with the product topology. The \textit{sofic shift} presented by $\mathcal{G}$ is the set
$$\mathcal{X} = \left\{x \in \mathcal{A}^\mathbb{Z} \mid x = \left\{\mathcal{L}(e_k)\right\}_{k\in \mathbb{Z}} \textrm{ for some } \{e_k\}_{k\in \mathbb{Z}} \in \mathcal{P}_\infty(\mathcal{G})\right\}$$
paired with the (left) shift map $\sigma: \mathcal{X} \rightarrow \mathcal{X}$, which maps $x = \{x_k\}_{k\in \mathbb{Z}}$ to $\sigma(x) = \{x_{k+1}\}_{k\in \mathbb{Z}}$. The graph $\mathcal{G}$ is called a \emph{presentation} of $\mathcal{X}$.

An element $x \in \mathcal{X}$ is a called a \emph{point}. A finite sequence $w=a_1a_2 \hdots a_n$ from $\mathcal{A}$ is called a \emph{word} of length $\lvert w \rvert=n$ if there exist $x \in \mathcal{X}$ and $k\in \mathbb{Z}$ such that $x_k x_{k+1}\hdots x_{k+n-1} = w$. Denote $\mathcal{B}(\mathcal{X})$ as the set of words in $\mathcal{X}$. The space $\mathcal{X}$ is said to be \emph{irreducible} if for all $u, v \in \mathcal{B}(\mathcal{X})$, either $uv \in \mathcal{B}(\mathcal{X})$ or there exists $w \in \mathcal{B}(\mathcal{X})$ such that $uwv \in  \mathcal{B}(\mathcal{X})$.

Any word in $\mathcal{X}$ is the label of a certain path in $\mathcal{G}$. It is possible to have more than one such paths for each word. This is similarly true for points as well. Besides that, the sofic shift itself may be presented by different labelled graphs. We will choose a presentation that has some useful properties.

A labelled graph is said to be \emph{right-resolving} if for every vertex, each outgoing edge has a distinct label. Any sofic shift has a presentation of this form. In fact, we can proceed further to obtain the \emph{minimal right-resolving presentation}, which is the one with the fewest vertices among all right-resolving presentations. For an irreducible sofic shift, this presentation is unique up to graph isomorphism.

A word $w \in \mathcal{B}(\mathcal{X})$ is said to be \emph{synchronising} in $\mathcal{G}$ if all paths with the label $w$ have the same terminal vertex. If $\mathcal{X}$ is irreducible and $\mathcal{G}$ is its minimal right-resolving presentation, then any word $u \in \mathcal{B}(\mathcal{X})$ can be extended into a synchronising word $uv$ in $\mathcal{G}$ for some $v \in \mathcal{B}(\mathcal{X})$. So in this case, the space $\mathcal{X}$ is guaranteed to have a synchronising word in $\mathcal{G}$.

Recall that a non-negative square matrix $A$ is irreducible if for every pair of indices $i$ and $j$, there exists $n \in \mathbb{N}$ such that the $ij$-entry of $A$, written as $(A^n)_{ij}$, is positive. In this case, its period is defined as $\gcd \{n \in \mathbb{N} \mid (A^n)_{ii}>0\}$ for any choice of index $i$. For a labelled graph $\mathcal{G}=(G, \mathcal{L})$, its adjacency matrix $A_\mathcal{G}$ is simply the adjacency matrix of the underlying graph $G$. The irreducibility and period of $\mathcal{G}$ is defined from its adjacency matrix $A_\mathcal{G}$.
A sofic shift is irreducible if and only if its minimal right-resolving presentation is also irreducible.

A right-resolving presentation is useful to determine the topological entropy of a sofic shift. For any choice of its right-resolving presentation $\mathcal{G}$, the topological entropy of $\mathcal{X}$ is given by $h\left(\mathcal{X}\right)=\ln \rho \left(A_\mathcal{G}\right)$, where $\rho \left(A_\mathcal{G}\right)$ denotes the spectral radius of the adjacency matrix $A_\mathcal{G}$.

The Artin-Mazur zeta function of a sofic shift is somehow complicated to be expressed in closed form. For this purpose, let $\mathcal{G}$ be a right-resolving presentation of $\mathcal{X}$ with the underlying graph $G=(V,E)$. For simplicity, we denote the vertex set as $V=\{1,2,\hdots,S\}$ where $S=\lvert V \rvert$. Observe that for each $a \in \mathcal{A}$ and $v \in V$, there is at most one outgoing edge from $v$ with label $a$. If such an edge exists, then we denote the terminal vertex as $v(a)$.

We construct a labelled graph $\mathcal{G}_j$ for $j=1,2,\hdots,S$ from $\mathcal{G}$. Its vertex set $V_j$ is the collection of subsets of $V$ with $j$ distinct vertices. For each subset $v^{(j)} \in V_j$, arrange the vertices in increasing order, i.e. $v^{(j)} = \{v_1,v_2,\hdots,v_j\}$ where $v_1 < v_2 < \hdots < v_j$. Now for each $a \in \mathcal{A}$, there is an edge from $v^{(j)}$ to $\tilde{v}^{(j)} = \{\tilde{v}_1, \tilde{v}_2, \hdots, \tilde{v}_j\}$ if all vertices $v_1(a), v_2(a), \hdots, v_j(a)$ are defined, distinct and contained in $\tilde{v}^{(j)}$. The label of this edge is $a$ if $\{v_1(a), v_2(a), \hdots, v_j(a)\}$ is an even permutation of $\tilde{v}^{(j)}$, while its label is $-a$ otherwise.

Define the \textit{signed subset matrix} $A_j$ from $\mathcal{G}_j$ as follows: for $v^{(j)},\tilde{v}^{(j)} \in V_j$, the $v^{(j)} \tilde{v}^{(j)}$-entry of $A_j$ is the number of edges with positive labels minus the number with negative labels from $v^{(j)}$ to $\tilde{v}^{(j)}$. Note that $\mathcal{G}_1$ is simply $\mathcal{G}$, so $A_1$ is the adjacency matrix $A_\mathcal{G}$.

From the setting above, the Artin-Mazur zeta function of $\mathcal{X}$ is given by
\begin{equation}\label{zeta function}
	\zeta_{\sigma}(z) = \prod_{j=1}^{S} \left(\det\left(I-zA_j\right)\right)^{(-1)^j}
\end{equation}
where $I$ is the identity matrix (of the same size to $A_j$). Since each determinant gives out a polynomial in terms of $z$, the zeta function is a rational function.

\section{Orbit Growth of Sofic Shifts}

Any sofic shift can be decomposed into several irreducible sofic subshifts, which can be studied separately (see \cite{Lind-Marcus}). So, it is sufficient to consider an irreducible sofic shift here. Furthermore, the topological entropy of an irreducible sofic shift is zero if and only if it is finite, and in fact, it is entirely a single closed orbit. To avoid triviality, we can also assume that the topological entropy is positive.

Let $\mathcal{X}$ be an irreducible sofic shift with positive topological entropy. Let $\mathcal{G}$ be its minimal right-resolving presentation. Our assumption implies that its adjacency matrix $A_\mathcal{G}$ is irreducible. Denote $p$ and $\lambda$ as the period and spectral radius of $A_\mathcal{G}$ respectively. The assumption also implies that $\lambda>1$. By Perron-Frobenius theory, there are exactly $p$ eigenvalues of $A_\mathcal{G}$ such that each has modulus $\lambda$, has the form $z_k=\lambda \cdot\exp \left(\frac{2k\pi}{p}\boldsymbol{i}\right)$ for $k=0,1,\hdots,p-1$, and is simple (see \cite{Lind-Marcus}).

Our aim is to apply Theorem \ref{theorem: zeta function approach} to deduce the orbit growth of $\mathcal{X}$. From (\ref{zeta function}), observe that
$$\zeta_{\sigma}(z) = \prod_{j=1}^S \prod_{\mu}(1-\mu z)^{(-1)^j}$$
where $\mu$ runs through non-zero eigenvalues of $A_j$. So, any zero or pole of $\zeta_{\sigma}(z)$ must be in the form $\mu^{-1}$. As per Theorem \ref{theorem: zeta function approach}, we will show that $\zeta_{\sigma}(z)$ can be written as
$$\zeta_{\sigma}(z) = \frac{\alpha(z)}{1-\lambda^p z^p}, \quad \alpha(z) = \prod_{j=1}^S \prod_{\substack{\mu \\ \lvert \mu\rvert < \lambda}}(1-\mu z)^{(-1)^j}$$
where $\alpha(z)$ is analytic and non-zero for $\vert z \rvert<R \lambda^{-1}$, for
$$R = \min_{\substack{\mu \\ \lvert \mu\rvert < \lambda}}\left\{\frac{\lambda}{\lvert \mu\rvert}\right\}.$$
Equivalently, we require that each $z_k^{-1}$ is a simple pole of $\zeta_{\sigma}(z)$, while all zeros and other poles are located beyond radius $\lambda^{-1}$. The former is true based on Perron-Frobenius theory above. For the latter, by comparing spectral radii, it is sufficient to prove that $\rho(A_j)< \lambda$ for $j \neq 1$.

Now for $j=2,3,\hdots,S$, define $\tilde{\mathcal{G}}_j$ as the labelled graph $\mathcal{G}_j$ but all negative labels become positive instead. Note that $\tilde{\mathcal{G}}_j$ is indeed right-resolving. Let $\tilde{A}_j$ be its adjacency matrix. Observe that $\lvert\left(A_j\right)_{uv}\rvert \leq (\tilde{A}_j)_{uv}$ for each $uv$-entry where $u,v \in V^{(j)}$, and hence $\rho(A_j) \leq \rho (\tilde{A}_j)$. We now prove that $\rho (\tilde{A}_j) < \lambda$ for $j \neq 1$.

\begin{lemma}
	$\rho (\tilde{A}_j) < \lambda$ for $j \neq 1$.
\end{lemma}
\begin{proof}
	Let $\tilde{\mathcal{X}}_j$ be the sofic shift presented by $\tilde{\mathcal{G}}_j$. It is easy to see that $\tilde{\mathcal{X}}_j \subseteq \mathcal{X}$. Indeed, any bi-infinite path in $\tilde{\mathcal{G}}_j$ gives $j$ corresponding bi-infinite paths in $\mathcal{G}$. All these paths have the same label.
	
	We claim that $\tilde{\mathcal{X}}_j \neq \mathcal{X}$ for $j \neq 1$. Take a synchronising word $w \in \mathcal{B}(\mathcal{X})$ in $\mathcal{G}$. Let $x \in \mathcal{X}$ be a point containing $w$, i.e. there exists $k \in \mathbb{Z}$ such that $x_k x_{k+1}\hdots x_{k+\lvert w \rvert-1} = w$. Suppose, for contradiction, that $x \in \tilde{\mathcal{X}}_j$. A bi-infinite path $\tilde{p}$ in $\tilde{\mathcal{G}}_j$ with label $x$ gives bi-infinite paths $p_1, p_2, \hdots, p_j$ in $\mathcal{G}$. Since $w$ is synchronising in $\mathcal{G}$, the edges at coordinate $k+\lvert w \rvert-1$ for the paths $p_1,p_2, \hdots,p_j$ have the same terminal vertex $v \in V$ in $\mathcal{G}$. Hence, the edge at the said coordinate for the path $\tilde{p}$ has the terminal vertex $v^{(j)}=\{v\} \in V_j$ in $\tilde{\mathcal{G}}_j$. This is a contradiction since any vertex in $\tilde{\mathcal{G}}_j$ is a subset of $j$ distinct vertices from $\mathcal{G}$. Overall, this shows that $x \not \in \tilde{\mathcal{X}}_j$ and consequently $\tilde{\mathcal{X}}_j \subsetneq \mathcal{X}$.
	
	Since $\tilde{\mathcal{X}}_j$ is a proper subshift of $\mathcal{X}$, the topology entropy of $\tilde{\mathcal{X}}_j$ is strictly less than that of $\mathcal{X}$ (see \cite{Lind-Marcus}). This implies that
	$$\ln \rho (\tilde{A}_j) = h(\tilde{\mathcal{X}}_j) < h(\mathcal{X}) = \ln \lambda$$
	and hence $\rho (\tilde{A}_j) < \lambda$.
\end{proof}

Tracing back the arguments above, we may now deduce the orbit growth of $\mathcal{X}$ from Theorem \ref{theorem: zeta function approach}.

\begin{theorem}\label{theorem: main result}
	Let $\mathcal{X}$ be an irreducible sofic shift with the topological entropy $\ln \lambda>0$. Suppose that its minimal right-resolving presentation has period $p$. Then,
	$$\pi_\sigma(N) \sim \frac{p \lambda^{p\left(\left\lfloor\frac{N}{p}\right\rfloor+1\right)}}{N(\lambda^p-1)},\quad \mathcal{M}_\sigma(N) \sim \frac{pe^{-\gamma}}{\alpha N} \quad\textrm{and}$$
	$$\mathscr{M}_\sigma(N) = \ln \left\lfloor\frac{N}{p}\right\rfloor + \gamma + \ln \alpha -C +o(1)$$
	where $\gamma$ is the Euler-Mascheroni constant, $\alpha$ is a positive constant specified by
	$$\alpha = \lim_{z \rightarrow \lambda^{-1}} (1-\lambda^pz^p)\cdot \zeta_{\sigma}(z)$$
	where $\zeta_{\sigma}(z)$ is the Artin-Mazur zeta function in (\ref{zeta function}), and $C$ is another positive constant specified in (\ref{equation: C}).
\end{theorem}

\section{Additional Remarks}

\subsection{Periods of a sofic shift and its presentation}

For an irreducible sofic shift, its period is defined as $\gcd \{n \in \mathbb{N} \mid F_\sigma(n)>0\}$, where $F_\sigma(n)$ is the number of periodic points of period $n$. This period must be a divisor of the period of the minimal right-resolving presentation, but they are not necessarily equal (see \cite{Lind-Marcus}). Notice that the period of the presentation (which is $p$) appears in the asymptotic results in Theorem \ref{theorem: main result}, instead of the period of the sofic shift itself. This is somewhat interesting that the period of the presentation is more prominent in determining the orbit growth.

\subsection{Orbit growth of a periodic-finite-type shift}

A periodic-finite-type shift (which includes a shift of finite type) is a sofic shift. We obtained its orbit growth in \cite{Nordin-Noorani_PFT} by using its Moision-Siegel presentation \cite{Beal,Manada-Kashyap}. However, our previous work was incomplete since the result was only applicable for the one with an irreducible such presentation. An irreducible periodic-finite-type shift may have a reducible Moision-Siegel presentation. Nonetheless, Theorem \ref{theorem: main result} above covers both cases, and it indeed agrees with our previous result.

\subsection{Applications to finite group and homogeneous extensions of a sofic shift}

Our result in Theorem \ref{theorem: main result} is crucial to prove Chebotarev-like theorems for finite group extensions of sofic shifts, which is similar to the case for shifts of finite type \cite{Parry-Pollicott,Noorani-Parry,Mohamed-Noorani}. We only provide a simple explanation here about this topic. Let $\mathcal{X}$ be a sofic shift. Consider a finite group $K$ and a function $\Psi:\mathcal{X} \rightarrow K$ which depends only on first two coordinates, i.e. if $x = \{x_k\}_{k \in \mathbb{Z}}$ and $y = \{y_k\}_{k \in \mathbb{Z}}$ in $\mathcal{X}$ satisfy $x_0x_1=y_0y_1$, then $\Psi(x)=\Psi(y)$. The finite group extension of $\mathcal{X}$ with $K$ is the product $\mathcal{X} \times K$ paired with a map $\hat{\sigma} : \mathcal{X} \times K \rightarrow \mathcal{X} \times K$ where $\hat{\sigma}(x,g) = \left(\sigma(x), \Psi(x)g\right)$. We also define a free action of $K$ on $X \times K$ by $h \cdot (x, g) = (x, gh)$ for any $h \in K$. Under these settings, each closed orbit $\tau$ of $\mathcal{X}$ is associated with a conjugacy class of $K$. This is called the Frobenius class of $\tau$, which is denoted as $[\tau]$. Indeed, the class $[\tau]$ is the conjugacy class for the element $\Psi\left(\sigma^{\lvert\tau\rvert-1}(x)\right)\Psi\left(\sigma^{\lvert\tau\rvert-2}(x)\right)\hdots\Psi\left(x\right) \in K$ for any choice of $x \in \tau$. Now, we can count the closed orbits having the same Frobenius class by using our orbit counting functions. In other words, for a conjugacy class $C$ of $K$, define $\pi_\sigma^C(N)$, $\mathcal{M}_\sigma^C(N)$ and $\mathscr{M}_\sigma^C(N)$ similarly but they sum over the closed orbits with the Frobenius class $C$. We can follow the same arguments in \cite{Parry-Pollicott,Noorani-Parry,Mohamed-Noorani} and apply our result here to obtain the asymptotic behaviours of these counting functions. For instance, if $\hat{\sigma}$ is topologically mixing on $\mathcal{X} \times K$, then
$$\pi_\sigma^C(N) \sim \frac{\lvert C \rvert}{\lvert K \rvert} \cdot \frac{\lambda^{N+1}}{N(\lambda-1)},\quad \mathcal{M}_\sigma^C(N) \sim \frac{e^{-\gamma \lvert C \rvert/\lvert K \rvert}}{\beta N^{\lvert C \rvert / \lvert K \rvert}} \quad\textrm{and}$$
$$\mathscr{M}_\sigma^C(N) = \frac{\lvert C \rvert}{\lvert K \rvert} \ln N + \frac{\lvert C \rvert}{\lvert K \rvert}\gamma + \ln \beta - \mathcal{C} +o(1)$$
where $\ln \lambda$ is the topological entropy of $\mathcal{X}$, $\gamma$ is the Euler-Mascheroni constant, and $\beta$ and $\mathcal{C}$ are positive constants. We may also prove similar theorems for finite homogeneous extensions of sofic shifts, and again, the arguments follow similarly as in \cite{Noorani-Parry,Mohamed-Noorani}.

\bmhead{Acknowledgments}

The authors would like to thank the Ministry of Higher Education, Malaysia (grant number FRGS/1/2019/STG06/UKM/01/3) for the financial support in this research.

\bmhead{Competing Interests}

The authors declare that they have no competing interest.

%%\bibliography{sn-bibliography}% common bib file
\input sn-article.bbl


\end{document}